\documentclass[12pt,notitlepage]{amsart}
\usepackage{latexsym,amsfonts,amssymb,amsmath,amsthm} 
\usepackage{color}
\newtheorem{cor}{Corollary}
\newtheorem{thm}{Theorem}
\newtheorem{lem}{Lemma}
\newtheorem{prop}{Proposition}
\newtheorem{conj}{Conjecture}
\newtheorem{defi}{Definition}

\theoremstyle{remark}
\newtheorem{rmk}{Remark}

\theoremstyle{plain}

\theoremstyle{remark}

\numberwithin{equation}{section}
\begin{document}

\title{The Mean Value of $L(\tfrac{1}{2},\chi)$ in the Hyperelliptic Ensemble}

\author{J. C. Andrade}
\address{School of Mathematics, University of Bristol, Bristol BS8 1TW, UK}
\email{j.c.andrade@bristol.ac.uk}
\thanks{JCA was supported by an Overseas Research Scholarship and an University of Bristol Research Scholarship.}

\author{J. P. Keating}
\address{School of Mathematics, University of Bristol, Bristol BS8 1TW, UK}
\email{j.p.keating@bristol.ac.uk}
\thanks{JPK is sponsored by the Air Force Office of Scientific Research, Air Force Material Command, USAF, under grant number FA8655-10-1-3088. The U.S. Government is authorized to reproduce and distribute reprints for Governmental purpose notwithstanding any copyright notation thereon.}


\begin{abstract}
We obtain an asymptotic formula for the first moment of quadratic Dirichlet $L$--functions over function fields at the central point $s=\tfrac{1}{2}$. Specifically, we compute the expected value of $L(\tfrac{1}{2},\chi)$ for an ensemble of hyperelliptic curves of genus $g$ over a fixed finite field as $g\rightarrow\infty$. Our approach relies on the use of the analogue of the approximate functional equation for such $L$--functions. The results presented here are the function field analogues of those obtained previously by Jutila in the number-field setting and are consistent with recent general conjectures for the moments of $L$--functions motivated by Random Matrix Theory.
\end{abstract}

\maketitle

\section{Introduction}
It is an important problem in analytic number theory to estimate moments of families of $L$--functions. For the classical Riemann zeta function the problem is to understand the asymptotic behaviour of
\begin{equation}
M_{k}(T)=\int_{0}^{T}|\zeta(\tfrac{1}{2}+it)|^{2k}dt.
\end{equation}

The leading order asymptotic for $M_{k}(T)$ is known just for $k=1$, due to Hardy and Littlewood \cite{HL}
\begin{equation}
M_{1}(T)\sim T\log T,
\end{equation} 
and for $k=2$, due to Ingham \cite{I}
\begin{equation}
M_{2}(T)\sim \frac{1}{2\pi^{2}}T\log^{4}T.
\end{equation}

For positive real $k$, it is conjectured that
\begin{equation}
M_{k}(T)\sim C_{k}T(\log T)^{k^{2}}
\end{equation}
for a positive constant $C_{k}$. Using random matrix theory, Keating and Snaith \cite{KeS1} conjectured a precise value for $C_{k}$ for all $k>0$ and more generally for $\mathfrak{R}(k)>-\tfrac{1}{2}$. The conjectures about moments of the Riemann zeta function and other $L$--functions were developed in \cite{KeS2}, \cite{CFKRS} and \cite{CFKRS2}, where they were extended to include all the principal lower order terms in the asymptotics; for example, when $k$ is a positive integer $\tfrac{1}{T}M_{k}(T)$ is given, conjecturally, as a polynomial of degree $k^{2}$ in $\log T$ with a reminder that vanishes like a power of $T$.

Considering the family of Dirichlet $L$--functions $L(s,\chi_{d})$, with $\chi_{d}$ a real primitive Dirichlet character modulo $d$ defined by the Jacobi symbol $\chi_{d}(n)=\left(\frac{d}{n}\right)$, the problem of mean values is to understand the asymptotic behavior of
\begin{equation}
\sum_{0<d\leq D}L(\tfrac{1}{2},\chi_{d})^{k},
\end{equation}
as $D\rightarrow\infty$. In this context Jutila \cite{J} proved that

\begin{multline}\label{eq:jutilaf}
\sum_{0<d\leq D}L(\tfrac{1}{2},\chi_{d})\\ =\frac{P(1)}{4\zeta(2)}D\left\{\log(D/\pi)+\frac{\Gamma^{'}}{\Gamma}(1/4)+4\gamma-1+4\frac{P^{'}}{P}(1)\right\}+O(D^{3/4+\varepsilon})
\end{multline}
where
$$P(s)=\prod_{p}\left(1-\frac{1}{(p+1)p^{s}}\right),$$
and
\begin{equation}
\sum_{0<d\leq D}L(\tfrac{1}{2},\chi_{d})^{2}=\frac{c}{\zeta(2)}D\log^{3}D+O(D(\log D)^{5/2+\varepsilon})
\end{equation}
with
$$c=\frac{1}{48}\prod_{p}\left(1-\frac{4p^{2}-3p+1}{p^{4}+p^{3}}\right).$$
Restricting $d$ to be odd, square--free and positive, so that $\chi_{8d}$ are real, primitive characters with conductor $8d$ and with $\chi_{8d}(-1)=1$, Soundararajan \cite{S} proved that
\begin{multline}
\frac{1}{D^{*}}\sideset{}{^*}\sum_{0<d\leq D}L(\tfrac{1}{2},\chi_{8d})^{3}\\
\sim\frac{1}{184320}\prod_{p\geq3}\left(1-\frac{12p^{5}-23p^{4}+23p^{3}-15p^{2}+6p-1}{p^{6}(p+1)}\right)(\log D)^{6},
\end{multline}
where the sum $\sum\nolimits^{*}$ over $d$ indicates that $d$ is odd and square--free, and $D^{*}$ is the number of such $d$ in $(0,D]$.  Note that \eqref{eq:jutilaf} includes all lower order terms in the sense of \cite{CFKRS}, in that the error term is $o(D)$.

Keating and Snaith \cite{KeS2} put forward the following conjecture for the leading order asymptotic for mean values of $L$--functions.

\begin{conj}
For $k$ fixed with $\mathfrak{R}(k)\geq0$, as $D\rightarrow\infty$
\begin{equation}
\frac{1}{D^{*}}\sideset{}{^{*}}\sum_{0<d\leq D}L(\tfrac{1}{2},\chi_{8d})^{k}\sim a_{k,Sp}\frac{G(k+1)\sqrt{\Gamma(k+1)}}{\sqrt{G(2k+1)\Gamma(2k+1)}}(\log D)^{k(k+1)/2}
\end{equation}
where
$$a_{k,Sp}=2^{-k(k+2)/2}\prod_{p\geq3}\frac{(1-\frac{1}{p})^{k(k+1)/2}}{1+\frac{1}{p}}\left(\frac{(1-\frac{1}{\sqrt{p}})^{-k}+(1+\frac{1}{\sqrt{p}})^{-k}}{2}+\frac{1}{p}\right)$$
and $G(z)$ is Barnes' $G$--function.
\end{conj}
Conjectures for the lower order terms are given in \cite{CFKRS}.  These conjectures coincide with the results listed above.

The main result of this paper, Theorem \ref{thm:mainthm}, can be seen as the function field analogue of Jutila's result \eqref{eq:jutilaf}, in that it constitutes an asymptotic formula for the first moment of quadratic $L$--functions over function fields that includes the lower order terms, in the sense of \cite{CFKRS}.

\section{Statement of Results}
Let $\mathbb{F}_{q}$ be a fixed finite field of odd cardinality and $A=\mathbb{F}_{q}[x]$ be the polynomial ring over $\mathbb{F}_{q}$ in the variable $x$. Let $C$ be any smooth, projective, geometrically connected curve of genus $g\geq1$ defined over the finite field $\mathbb{F}_{q}$. The zeta function of the curve $C$ was introduced by Artin \cite{A} and is defined as

\begin{equation}
Z_{C}(u):=\exp\left(\sum_{n=1}^{\infty}N_{n}(C)\frac{u^{n}}{n}\right), \ \ \ \ \ |u|<1/q
\end{equation}
where $N_{n}(C):=\mathrm{Card}(C(\mathbb{F}_{q}))$ is the number of points on $C$ with coordinates in a field extension $\mathbb{F}_{q^{n}}$ of $\mathbb{F}_{q}$ of degree $n\geq1$. It was shown by Weil \cite{W} that the zeta function associated to $C$ is a rational function of the form

\begin{equation}\label{eq:zetaC}
Z_{C}(u)=\frac{P_{C}(u)}{(1-u)(1-qu)},
\end{equation}
where $P_{C}(u)\in\mathbb{Z}[u]$ is a polynomial of degree $2g$, with $P_{C}(0)=1$, and that it satisfies the functional equation

\begin{equation}\label{eq:funceq}
P_{C}(u)=(qu^{2})^{g}P_{C}\left(\frac{1}{qu}\right).
\end{equation}
By the Riemann Hypothesis for curves over finite fields, proved by Weil \cite{W}, one knows that the zeros of $P_{C}(u)$ all lie on the circle $|u|=q^{-1/2}$, i.e.,
$$P_{C}(u)=\prod_{j=1}^{2g}(1-\alpha_{j}u), \ \ \ \ \ \mathrm{with} \ \ |\alpha_{j}|=\sqrt{q} \ \ \mathrm{for \ all}\ j.$$

Our main goal is to establish an asymptotic formula for the average value of $P_{C}(u)$ at the central point $u=1/\sqrt{q}$ as we vary $C$ in a family of hyperelliptic curves of increasing genus $g$ defined over $\mathbb{F}_{q}$ where $q$ is fixed and assumed to be odd. To establish the result we choose the particular family, denoted by $\mathcal{H}_{2g+1,q}$, of all hyperellipic curves given in affine form by

$$C_{D}:y^{2}=D(x)$$
where

$$D(x)=x^{2g+1}+a_{2g}x^{2g}+\cdots+a_{1}x+a_{0}\in\mathbb{F}_{q}[x]$$
is a square--free, monic polynomial of degree $2g+1$. The curve $C_{D}$ is thus nonsingular and of genus $g$ and the family is denoted by
$$\mathcal{H}_{2g+1,q}=\{D \ \mathrm{monic}, \ \mathrm{deg}(D)=2g+1, \ D \ \mathrm{square\ free}, \ D\in\mathbb{F}_{q}[x]\}.$$

In this paper we establish an asymptotic formula for

\begin{equation}\label{eq:2.4}
\sum_{D\in\mathcal{H}_{2g+1,q}}P_{C_{D}}(q^{-1/2})
\end{equation}
as $g\rightarrow\infty$. 

\begin{thm}
\label{thm:mainthm}
Let $q$ be the fixed cardinality of the ground field $\mathbb{F}_{q}$ and assume for simplicity that $q\equiv 1\pmod{4}$. Then

\begin{multline}
\sum_{D\in\mathcal{H}_{2g+1,q}}P_{C_{D}}(q^{-1/2})\\
=\frac{P(1)}{2\zeta_{A}(2)}|D|\left\{\log_{q}|D|+1+\frac{4}{\log q}\frac{P^{'}}{P}(1)\right\}+O\left(|D|^{3/4+\frac{\log_{q}2}{2}}\right),
\end{multline}
where
\begin{equation}
P(s)=\prod_{\substack{P \ \mathrm{monic} \\ \mathrm{irreducible}}}\left(1-\frac{1}{(|P|+1)|P|^{s}}\right),
\end{equation}
$|f|=q^{\mathrm{deg}f}$
for any polynomial $f\in\mathbb{F}_{q}[x]$ (so $|D|=q^{2g+1}$), and
\begin{equation}
\zeta_{A}(s)=\frac{1}{1-q^{1-s}}
\end{equation}
is the zeta function associated to $A=\mathbb{F}_{q}[x]$.
\end{thm}
Comparing \eqref{eq:jutilaf} with Theorem \ref{thm:mainthm}, one sees clearly the analogy between function fields and the number-field result.

\begin{cor}\label{cor:1}
Under the same assumptions of Theorem \ref{thm:mainthm} we have,

\begin{equation}
\frac{1}{\#\mathcal{H}_{2g+1,q}}\sum_{D\in\mathcal{H}_{2g+1,q}}P_{C_{D}}(q^{-1/2})\sim \frac{1}{2}P(1)(\log_{q}|D|)=\frac{1}{2}P(1)(2g+1)
\end{equation}
as $g\rightarrow\infty$.
\end{cor}

It seems likely that the calculations presented in this paper can be extended to establish the corresponding asymptotic formula for the second power moment of Dirichlet $L$--functions over the rational function field, and possibly for the third power moment also, in the same way that for the classical quadratic $L$--functions Jutila established the second moment and Soundararajan the third moment.

Previously, J. Hoffstein and M. Rosen \cite{HR} obtained an asymptotic formula for the first moment of Dirichlet $L$--functions over function fields making use of Eisentein series for the metaplectic two--fold cover of $GL(2, k_ {\infty})$.  They considered the sum over all square-free polynomials of a prescribed degree. One important difference between Hoffstein and Rosen's result and ours is that we sum over square--free and monic polynomials, which means that we are averaging over positive and fundamental discriminants in this setting.  The two results have the same general form, but are different in their details.  Our calculation is complementary to that developed in \cite{HR}, being more similar to the classical methods employed in \cite{J}.

In a recent paper A. Bucur and A. Diaconu \cite{BucurMMJ} established the following result:

\begin{thm}[Bucur, Diaconu]\label{thm:BD}As $q\rightarrow\infty$, we have

$$\sum_{\substack{d\in A \\ d \ \mathrm{monic} \\ \mathrm{deg}(d)=2g}}L(\tfrac{1}{2},\chi_{d})^{4}\sim\frac{g(1+g)^{2}(2+g)^{2}(3+g)(1+2g)(3+2g)^{2}(5+2g)}{75600}q^{2g}.$$

\end{thm}

Theorem \ref{thm:BD} is the fourth power moment for Quadratic Dirichlet $L$--Functions as $q\rightarrow\infty$ and $g$ is fixed. This is therefore the opposite limit to that which we consider in this paper. To establish the result above Bucur and Diaconu make use of Multiple Dirichlet Series and the Weyl group action of a particular Kac-Moody algebra. 

The calculations presented here to establish an asymptotic formula for the first moment of quadratic Dirichlet $L$--functions over the rational function field are based on elementary estimates and techniques which are in the spirit of those used by Faifman and Rudnick \cite{FR}, Kurlberg and Rudnick \cite{KR} and Rudnick \cite{Ru}.  Specifically, we make use of the analogue of the approximate functional equation for $L$--functions over function fields. It is important to emphasize that the limit we consider here is $q$ is fixed and $g\rightarrow\infty$, rather than the limit $q\rightarrow\infty$ and $g$ fixed. In the latter case Katz and Sarnak \cite{KS1, KS2} established that the conjugacy classes $\{\Theta_{C}:C\in\mathcal{H}_{2g+1}\}$ become uniformly distributed in $USp(2g)$ in the limit $q\rightarrow\infty$, and so the determination of the moments of $P_{C}(u)$ becomes a purely Random Matrix Theory calculation. In this context we note explicitly that in many of our estimates (e.g. when we use the $O$ and $\ll$ notations) the implied constant may depend on $q$.

\section{Preliminaries on Quadratic $L$--Functions and Dirichlet Characters for Function Fields.}

We begin by presenting some background on the zeta functions associated with hyperelliptic curves. The theory was initiated by E. Artin \cite{A}.  For a general reference, see \cite{Ro}.

\subsection{Basic facts about $\mathbb{F}_{q}[x]$}
We define the norm of a polynomial $f\in\mathbb{F}_{q}[x]$ in the following way. For $f\neq0$, set $|f|:=q^{\mathrm{deg}f}$ and if $f=0$, set $|f|=0$. A monic irreducible polynomial is called a ``prime" polynomial.

The zeta function of $A=\mathbb{F}_{q}[x]$, denoted by $\zeta_{A}(s)$, is defined by the infinite series

\begin{equation}\label{eq:zetaA}
\zeta_{A}(s):=\sum_{\substack{f\in A \\ f \ \mathrm{monic}}}\frac{1}{|f|^{s}}=\prod_{\substack{P \ \mathrm{monic} \\ \mathrm{irreducible}}}\left(1-|P|^{-s}\right)^{-1}, \ \ \ \ \ \ \mathfrak{R}(s)>1
\end{equation}
which is
\begin{equation}\label{eq:zetaA1}
\zeta_{A}(s)=\frac{1}{1-q^{1-s}}.
\end{equation}
We can also define the analogue of the Mobius function $\mu(f)$ and the Euler totient function $\Phi(f)$ for $A=\mathbb{F}_{q}[x]$ as follows:
\begin{equation}\label{eq:3.3}
\mu(f)=\left\{
\begin{array}{rcl}
(-1)^{t}, & f=\alpha P_{1}P_{2}\ldots P_{t},\\
0, & \mathrm{otherwise},\\
\end{array}
\right.
\end{equation}
where each $P_{j}$ is a distinct monic irreducible, and
\begin{equation}
\Phi(f)=\sum_{\substack{g \ \mathrm{monic} \\ \mathrm{deg}(g)<\mathrm{deg}(f) \\ (f,g)=1}}1.
\end{equation}

\subsection{Quadratic Characters and the Corresponding $L$--function}
Assume that $q$ is odd and let $P(x)\in\mathbb{F}_{q}[x]$ be an irreducible polynomial. Then by \cite[Proposition 1.10]{Ro} if $f\in A$ and $P\nmid f$ we know that the congruence $x^{d}\equiv f\pmod P$ is solvable if and only if
$$f^{\frac{|P|-1}{d}}\equiv1\pmod P,$$
where $d$ is a divisor of $q-1$. The quadratic residue symbol arises when we consider $d=2$ and is denoted by $(f/P)\in\{\pm1\}$:
\begin{equation}
\left(\frac{f}{P}\right)\equiv f^{(|P|-1)/2}\pmod P,
\end{equation}
for $f$ coprime to $P$. 

We also can then define the Jacobi symbol $(f/Q)$ for arbitrary monic $Q$. Let $f$ be coprime to $Q$ and $Q=\alpha P_{1}^{e_{1}}P_{2}^{e_{2}}\ldots P_{s}^{e_{s}}$ so
$$\left(\frac{f}{Q}\right)=\prod_{j=1}^{s}\left(\frac{f}{P_{j}}\right)^{e_{j}}.$$
If $f,Q$ are not coprime we set $(f/Q)=0$ and if $\alpha\in\mathbb{F}_{q}^{*}$ is a scalar then

$$\left(\frac{\alpha}{Q}\right)=\alpha^{((q-1)/2)\mathrm{deg}Q}.$$

The analogue of the quadratic reciprocity law for function fields is

\begin{thm}[Quadratic reciprocity]

Let $A,B\in\mathbb{F}_{q}[x]$ be relatively prime and $A\neq0$ and $B\neq0$. Then,

$$\left(\frac{A}{B}\right)=\left(\frac{B}{A}\right)(-1)^{((q-1)/2)\mathrm{deg}(A)\mathrm{deg}(B)}=\left(\frac{B}{A}\right)(-1)^{((|A|-1)/2)((|B|-1)/2)}$$

\end{thm}

\begin{defi}
Let $D\in\mathbb{F}_{q}[x]$ be square-free.  We define the \textit{quadratic character} $\chi_{D}$ using the quadratic residue symbol for $\mathbb{F}_{q}[x]$ by
\begin{equation}
\chi_{D}(f)=\left(\frac{D}{f}\right).
\end{equation}
So, if $P\in A$ is monic irreducible we have
$$\chi_{D}(P)=\left\{
\begin{array}{cl}
0, & \mathrm{if}\ P\mid D,\\
1, & \mathrm{if}\ P\not{|} D \ \mathrm{and} \ D \ \mathrm{is \ a \ square \ modulo} \ P,\\
-1, & \mathrm{if}\ P\not{|} D \ \mathrm{and} \ D \ \mathrm{is \ a \ non\ square \ modulo} \ P.\\
\end{array}
\right.$$
\end{defi}
We define the $L$--function corresponding to the quadratic character $\chi_{D}$ by
\begin{equation}
\mathcal{L}(u,\chi_{D}):=\prod_{\substack{P \ \mathrm{monic}\\ \mathrm{irreducible}}}(1-\chi_{D}(P)u^{\mathrm{deg}P})^{-1}, \ \ \ \ \ |u|<1/q
\end{equation}
where $u=q^{-s}$. The $L$--function above can also be expressed as an infinite series in the usual way:
\begin{equation}\label{eq:3.8}
\mathcal{L}(u,\chi_{D})=\sum_{\substack{f\in A \\ f \ \mathrm{monic}}}\chi_{D}(f)u^{\mathrm{deg}f}=L(s,\chi_{D})=\sum_{\substack{f\in A \\ f \ \mathrm{monic}}}\frac{\chi_{D}(f)}{|f|^{s}}.
\end{equation}
We can write \eqref{eq:3.8} as
\begin{equation}\label{eq:3.9}
\mathcal{L}(u,\chi_{D})=\sum_{n\geq0}\sum_{\substack{\mathrm{deg}(f)=n \\ f \ \mathrm{monic}}}\chi_{D}(f)u^{n}.
\end{equation}
If we denote
$$A_{D}(n):=\sum_{\substack{f \ \mathrm{monic} \\ \mathrm{deg}(f)=n}}\chi_{D}(f),$$
we can write \eqref{eq:3.9} as
\begin{equation}
\sum_{n\geq0}A_{D}(n)u^{n},
\end{equation}
and \cite[Propostion 4.3]{Ro}, if $D$ is a non--square polynomial of positive degree, then $A_{D}(n)=0$ for $n\geq\mathrm{deg}(D)$.  So in this case the $L$--function is in fact a polynomial of degree at most $\mathrm{deg}(D)-1$.

We now assume the primitivity condition that $D$ is a square--freee monic polynomial of positive degree. Following the arguments presented in \cite{Ru} we have that $\mathcal{L}(u,\chi_{D})$ has a ``trivial" zero at $u=1$ if and only if $\mathrm{deg}(D)$ is even, which enables us to define the ``completed" $L$--function
\begin{equation}\label{eq:3.11}
\mathcal{L}(u,\chi_{D})=(1-u)^{\lambda}\mathcal{L}^{*}(u,\chi_{D}), \ \ \ \ \ \lambda=\left\{
\begin{array}{rcl}
1, & \mathrm{deg}(D) \ \mathrm{even},\\
0, & \mathrm{deg}(D) \ \mathrm{odd},\\
\end{array}
\right.
\end{equation}
where $\mathcal{L}^{*}(u,\chi_{D})$ is a polynomial of even degree
$$2\delta=\mathrm{deg}(D)-1-\lambda$$
satisfying the functional equation
$$\mathcal{L}^{*}(u,\chi_{D})=(qu^{2})^{\delta}\mathcal{L}^{*}(1/qu,\chi_{D}).$$

By \cite[Proposition 14.6 and 17.7]{Ro}, $\mathcal{L}^{*}(u,\chi_{D})$ is the Artin $L$--function corresponding to the unique nontrivial quadratic character of $\mathbb{F}_{q}(x)(\sqrt{D(x)})$.  The fact that is important for this paper is that the numerator $P_{C}(u)$ of the zeta-function of the hyperelliptic curve $y^{2}=D(x)$ coincides with the completed Dirichlet $L$--function $\mathcal{L}^{*}(u,\chi_{D})$ associated with the quadratic character $\chi_{D}$, as was found in Artin's thesis. So we can write $\mathcal{L}^{*}(u,\chi_{D})$ as
\begin{equation}
\mathcal{L}^{*}(u,\chi_{D})=\sum_{n=0}^{2\delta}A^{*}_{D}(n)u^{n},
\end{equation}
where $A_{D}^{*}(0)=1$ and $A_{D}^{*}(2\delta)=q^{\delta}$.

For $D$ monic, square-free, and of positive degree, the zeta function \eqref{eq:zetaC} of the hyperelliptic curve $y^{2}=D(x)$ is
\begin{equation}\label{eq:3.13}
Z_{C_{D}}(u)=\frac{\mathcal{L}^{*}(u,\chi_{D})}{(1-u)(1-qu)}.
\end{equation}
As we are interested in computing
\begin{equation}\label{eq:3.14}
\frac{1}{\#\mathcal{H}_{2g+1,q}}\sum_{D\in\mathcal{H}_{2g+1,q}}P_{C_{D}}(q^{-1/2}),
\end{equation}
where $D\in\mathcal{H}_{2g+1,q}$, we have that $\mathrm{deg}(D)$ is odd and so by \eqref{eq:3.11} we have that $\mathcal{L}^{*}(u,\chi_{D})=\mathcal{L}(u,\chi_{D})$ and \eqref{eq:3.14} becomes
\begin{equation}\label{eq:3.15}
\frac{1}{\#\mathcal{H}_{2g+1,q}}\sum_{D\in\mathcal{H}_{2g+1,q}}\mathcal{L}(q^{-1/2},\chi_{D}),
\end{equation}
which will be the principal quantity of study. So the principal problem we consider is to obtain an asymptotic formula
for
\begin{equation}
\sum_{D\in\mathcal{H}_{2g+1,q}}\mathcal{L}(q^{-1/2},\chi_{D})
\end{equation}
as $g\rightarrow\infty$, where $\mathcal{L}(u,\chi_{D})$ is the Dirichlet $L$--function associated with the quadratic character $\chi_{D}$ of $\mathbb{F}_{q}[x]$.

\subsection{The Hyperelliptic Ensemble $\mathcal{H}_{2g+1,q}$}
Let $\mathcal{H}_{d}$ be the set of square--free monic polynomials of degree $d$ in $\mathbb{F}_{q}[x]$. The cardinality of $\mathcal{H}_{d}$ is
$$\#\mathcal{H}_{d}=\left\{
\begin{array}{lcl}
(1-1/q)q^{d}, & d\geq2,\\
q, & d=1.\\
\end{array}
\right.$$
(This can be proved using
$$\sum_{d>0}\frac{\#\mathcal{H}_{d}}{q^{ds}}=\sum_{\substack{f \ \mathrm{monic} \\ \mathrm{squarefree}}}|f|^{-s}=\frac{\zeta_{A}(s)}{\zeta_{A}(2s)}$$
and \eqref{eq:zetaA1}  \cite[Proposition 2.3]{Ro}). In particular, for $g\geq1$ we have,
\begin{equation}\label{eq:3.17}
\#\mathcal{H}_{2g+1,q}=(q-1)q^{2g}=\frac{|D|}{\zeta_{A}(2)}.
\end{equation}

We can treat $\mathcal{H}_{2g+1,q}$ as a probability space (ensemble) with uniform probability measure. Thus the expected value of any continuous function $F$ on $\mathcal{H}_{2g+1,q}$ is defined as
\begin{equation}
\left\langle F(D)\right\rangle:=\frac{1}{\#\mathcal{H}_{2g+1,q}}\sum_{D\in\mathcal{H}_{2g+1,q}}F(D).
\end{equation}

Using the Mobius function $\mu$ of $\mathbb{F}_{q}[x]$ defined in \eqref{eq:3.3} we can sieve out the square-free polynomials, since
\begin{equation}
\sum_{A^{2}|D}\mu(A)=\left\{
\begin{array}{rcl}
1, & D \ \mathrm{square \ free},\\
0, & \mathrm{otherwise}.\\
\end{array}
\right.
\end{equation}
And in this way we can write the expected value of any function $F$ as
\begin{eqnarray}\label{eq:3.20}
\left\langle F(D)\right\rangle &=& \frac{1}{\#\mathcal{H}_{2g+1,q}}\sum_{\substack{D \ \mathrm{monic} \\ \mathrm{deg}(D)=2g+1}}\sum_{A^{2}\mid D}\mu(A)F(D)\\
& = & \frac{1}{(q-1)q^{2g}}\sum_{2\alpha+\beta=2g+1}\sum_{\substack{B \ \mathrm{monic} \\ \mathrm{deg}B=\beta}}\sum_{\substack{A \ \mathrm{monic} \\ \mathrm{deg}A=\alpha}}\mu(A)F(A^{2}B).\nonumber
\end{eqnarray}

\subsection{Spectral Interpretation and the Katz--Sarnak Philosophy}
As already noted, the goal of this paper is to explore the limit $q$ fixed and $g\rightarrow\infty$, as in \cite{FR, KR, Ru}. It is worth explaining an important difference from the opposite limit in which $q\rightarrow\infty$ with $g$ fixed.

The Riemann Hypothesis for curves over a finite field, proved by Weil \cite{W}, is that all zeros of $Z_{C}(u)$, and hence $\mathcal{L}^{*}(u,\chi_{D})$, lie on the circle $|u|=q^{-1/2}$. Equivalently, all the roots of $L(s,\chi_{D})$ lie on the line $\mathfrak{R}(s)=\tfrac{1}{2}$.

The polynomial $\mathcal{L}^{*}(u,\chi_{D})$ is the characteristic polynomial of a unitary symplectic matrix $\Theta_{C_{D}}\in USp(2g)$, defined up to conjugacy, and we can write
$$\mathcal{L}^{*}(u,\chi_{D})=\det(I-u\sqrt{q}\Theta_{C_{D}}).$$
The eigenvalues of $\Theta_{C_{D}}$ are of the form $e(\theta_{C,j})$, $j=1,\ldots,2g$, where $e(\theta)=e^{2\pi i\theta}$.

For a fixed genus $g$, Katz and Sarnak \cite{KS1} showed that the conjugacy classes (Frobenius classes) $\{\Theta_{C_{D}}:C_{D}\in\mathcal{H}_{2g+1,q}\}$ become equidistributed (with respect to Haar measure) in the unitary symplectic group $USp(2g)$  in the limit $q\rightarrow\infty$. That is, for any continuous function on the space of conjugacy classes of $USp(2g)$,

$$\lim_{q\rightarrow\infty}\left\langle F(\Theta_{C_{D}})\right\rangle=\int_{USp(2g)}F(A)dA$$
where $dA$ is the Haar measure.

This result allows one to compute arithmetic quantities such as $\log(\mathcal{L}^{*}(u,\chi_{D}))$ and the moments of $\mathcal{L}^{*}(u,\chi_{D})$ as $C_{D}$ varies in $\mathcal{H}_{2g+1,q}$ by using the corresponding computation in Random Matrix Theory for $USp(2g)$. For example, setting $u=q^{-1/2}$ one has in general that

$$\lim_{q\rightarrow\infty}\frac{1}{\#\mathcal{H}_{2g+1,q}}\sum_{D\in\mathcal{H}_{2g+1,q}}(\mathcal{L}^{*}(q^{-1/2},\chi_{D}))^s=\int_{USp(2g)}(\det(I-A))^sdA,$$
and Keating and Snaith \cite{KeS2} computed the moments of the characteristic polynomial in $USp(2g)$:

$$\int_{USp(2g)}\det(I-A)^{s}dA=2^{2gs}\prod_{j=1}^{g}\frac{\Gamma(1+g+j)\Gamma(1/2+s+j)}{\Gamma(1/2+j)\Gamma(1+s+g+j)}.$$
In our case here, where $q$ is fixed and $g\rightarrow\infty$ the matrices $\Theta_{C_{D}}$ inhabit different spaces as $g$ grows, and we do not know how to formulate an equidistribution problem. 


\subsection{``Approximate" Functional Equation}
The starting point in the proof of Theorem \ref{thm:mainthm} is a representation for $\mathcal{L}(u,\chi_{D})$ which can be viewed as the analogue of the approximate functional equation for the Riemann zeta function (equation 4.12.4 in \cite{T}) or for the quadratic Dirichlet $L$--function (Lemma 3 in \cite{J}). In our case the formula is an identity rather than an approximation.

\begin{lem}[``Approximate" Functional Equation]\label{lem:funceq} Let $\chi_{D}$ be a quadratic character, where $D\in\mathcal{H}_{2g+1,q}$. Then

\begin{multline}
P_{C_{D}}(q^{-1/2})=\mathcal{L}(q^{-1/2},\chi_{D})\\
=\sum_{n=0}^{g}\sum_{\substack{f_{1} \ \mathrm{monic} \\ \mathrm{deg}(f_{1})=n}}\chi_{D}(f_{1})q^{-n/2}+\sum_{m=0}^{g-1}\sum_{\substack{f_{2} \ \mathrm{monic} \\ \mathrm{deg}(f_{2})=m}}\chi_{D}(f_{2})q^{-m/2}.
\end{multline}
\end{lem}

\begin{proof}
Following \cite{CFKRS} we substitute $P_{C}(u)=\sum_{n=0}^{2g}a_{n}u^{n}$ into the functional equation \eqref{eq:funceq}
\begin{eqnarray}
\sum_{n=0}^{2g}a_{n}u^{n} & = & q^{g}u^{2g}\sum_{m=0}^{2g}a_{m}\left(\frac{1}{qu}\right)^{m}=q^{g}u^{2g}\sum_{m=0}^{2g}a_{m}q^{-m}u^{-m}\nonumber \\
& = & \sum_{m=0}^{2g}a_{m}q^{g-m}u^{2g-m}=\sum_{k=0}^{2g}a_{2g-k}q^{k-g}u^{k}.\nonumber
\end{eqnarray}
Therefore,
$$\sum_{n=0}^{2g}a_{n}u^{n}=\sum_{k=0}^{2g}a_{2g-k}q^{k-g}u^{k}.$$
Equating coefficients we have that
$$a_{n}=a_{2g-n}q^{n-g} \ \ \ \ \ \mathrm{or} \ \ \ \ \ a_{2g-n}=a_{n}q^{g-n}$$
and so we can write the polynomial $P_{C}(u)$ as
\begin{eqnarray}
\sum_{n=0}^{2g}a_{n}u^{n} & = & \sum_{n=0}^{g}a_{n}u^{n}+\sum_{m=0}^{g-1}a_{2g-m}u^{2g-m}\nonumber \\
& = & \sum_{n=0}^{g}a_{n}u^{n}+\sum_{m=0}^{g-1}a_{m}q^{g-m}u^{2g-m}\nonumber\\
& = & \sum_{n=0}^{g}a_{n}u^{n}+q^{g}u^{2g}\sum_{m=0}^{g-1}a_{m}q^{-m}u^{-m}\label{eq:3.22}.
\end{eqnarray}
Writing $\displaystyle{a_{n}=\sum_{\substack{f \ \mathrm{monic} \\ \mathrm{deg}(f)=n}}\chi_{D}(f)}$ and $u=q^{-1/2}$ in \eqref{eq:3.22} proves the lemma.
\end{proof}
We can write the polynomial $P_{C}(u)$ using the variable $s$ and so \eqref{eq:3.22} becomes
\begin{equation}
\mathcal{L}(u,\chi_{D})=L(s,\chi_{D})=\sum_{\substack{f_{1} \ \mathrm{monic}\\ \mathrm{deg}(f_{1})\leq g}}\frac{\chi_{D}(f_{1})}{|f_{1}|^{s}}+(q^{1-2s})^{g}\sum_{\substack{f_{2} \ \mathrm{monic}\\ \mathrm{deg}(f_{2})\leq g-1}}\frac{\chi_{D}(f_{2})}{|f_{2}|^{1-s}}.
\end{equation}
\subsection{Two Simple Lemmas}
We will state two simple lemmas which will be used in the calculations below. The proofs can be found in \cite[Propositions 1.7 and 2.7]{Ro}.

\begin{lem}\label{lem:2}
\begin{equation}
\Phi(f)=|f|\prod_{\substack{P \ \mathrm{monic} \\ \mathrm{irreducible} \\ P|f}}\left(1-\frac{1}{|P|}\right).
\end{equation}
\end{lem}

\begin{lem}\label{lem:3}
\begin{equation}
\sum_{\substack{f \ \mathrm{monic} \\ \mathrm{deg}(f)=n}}\Phi(f)=q^{2n}(1-q^{-1}).
\end{equation}
\end{lem}

\section{Setting Up The Problem}
The basic quantity of study \eqref{eq:2.4} can be viewed, from \eqref{eq:3.13} and \eqref{eq:3.15}, as
\begin{equation}\label{eq:4.1}
\sum_{D\in\mathcal{H}_{2g+1,q}}\mathcal{L}(q^{-1/2},\chi_{D})=\sum_{D\in\mathcal{H}_{2g+1,q}}\sum_{n=0}^{2g}\sum_{\substack{f \ \mathrm{monic} \\ \mathrm{deg}(f)=n}}\chi_{D}(f)q^{-n/2}.
\end{equation}
Using Lemma \ref{lem:funceq} we can save $g$ terms and write \eqref{eq:4.1} as
\begin{multline}\label{eq:4.2}
\sum_{D\in\mathcal{H}_{2g+1,q}}\mathcal{L}(q^{-1/2},\chi_{D})\\ =\sum_{D\in\mathcal{H}_{2g+1,q}}\sum_{n=0}^{g}\sum_{\substack{f_{1} \ \mathrm{monic} \\ \mathrm{deg}(f_{1})=n}}\chi_{D}(f_{1})q^{-n/2}+\sum_{D\in\mathcal{H}_{2g+1,q}}\sum_{m=0}^{g-1}\sum_{\substack{f_{2} \ \mathrm{monic} \\ \mathrm{deg}(f_{2})=m}}\chi_{D}(f_{2})q^{-m/2}.
\end{multline}
As both terms on the right--hand side of \eqref{eq:4.2} are similar we need only worry about computing one of them to obtain the final result.

\subsection{Averaging the Approximate Functional Equation}
We are interested in obtaining an asymptotic formula for the first term on the RHS of \eqref{eq:4.2} and so we need to compute
\begin{multline}\label{eq:4.3}
\sum_{D\in\mathcal{H}_{2g+1,q}}\sum_{n=0}^{g}\sum_{\substack{f \ \mathrm{monic} \\ \mathrm{deg}(f)=n}}\chi_{D}(f)q^{-n/2}\\
=\sum_{n=0}^{g}q^{-n/2}\sum_{D\in\mathcal{H}_{2g+1,q}}\sum_{\substack{f \  \mathrm{monic}\\ \mathrm{deg}(f)=n \\ f=\Box=l^{2}}}\chi_{D}(f)+\sum_{n=0}^{g}q^{-n/2}\sum_{D\in\mathcal{H}_{2g+1,q}}\sum_{\substack{f \  \mathrm{monic}\\ \mathrm{deg}(f)=n \\ f\neq\Box}}\chi_{D}(f)\ \ \ \ \ \ \ \ \ \ \ \ \ \ \ \ \ \ \ \ \ \ \ \ \ \ \ \ \ \ \ \ \ \ \ \ \ \ \ \ \ \ \ \ \ \ \ \ \ \ \ \ \ \ \ \ \ \ \ \ \ \ \ \ \ \ \ \ \ \ \ \ \\
=\sum_{\substack{n=0 \\ 2|n}}^{g}q^{-n/2}\sum_{\substack{l \ \mathrm{monic} \\ \mathrm{deg}(l)=n/2}}\sum_{D\in\mathcal{H}_{2g+1,q}}\chi_{D}(l^{2})+\sum_{n=0}^{g}q^{-n/2}\sum_{\substack{f \  \mathrm{monic}\\ \mathrm{deg}(f)=n \\ f\neq\Box}}\sum_{D\in\mathcal{H}_{2g+1,q}}\chi_{D}(f) \ \ \ \ \ \ \ \ \ \ \ \ \ \ \ \ \ \ \ \ \ \ \ \ \ \ \ \ \ \ \ \ \ \ \ \ \ \ \ \ \ \ \ \ \ \ \ \ \ \ \ \ \ \ \ \\
=\sum_{\substack{n=0 \\ 2|n}}^{g}q^{-n/2}\sum_{\substack{l \ \mathrm{monic} \\ \mathrm{deg}(l)=n/2}}\sum_{\substack{D\in\mathcal{H}_{2g+1,q} \\ (D,l)=1}}1+\sum_{n=0}^{g}q^{-n/2}\sum_{\substack{f \  \mathrm{monic}\\ \mathrm{deg}(f)=n \\ f\neq\Box}}\sum_{D\in\mathcal{H}_{2g+1,q}}\chi_{D}(f), \ \ \ \ \ \ \ \ \ \ \ \ \ \ \ \ \ \ \ \ \ \ \ \ \ \ \ \ \ \ \ \ \ \ \ \ \ \ \ \ \ \ \ \ \ \ \ \ \ \ \ \ \ \ \ \ \ \ \ \ \ \ \ \ \ \ \ \ \ \ \ \ \ \ \ \ \ \ \ \ \ \ \ \ \ \ \ \ \ \ \ \ \ \ \ \ \ \ \ \ \ \ \ \ \ \ \ \ \ \
\end{multline}
where the first term on the RHS of the final expression corresponds to contributions of squares to the average and the second term to the contributions of non--squares.

Basically the problem is the following: for the square contributions we need to count square--free polynomials which are coprime to a fixed monic polynomial and to perform the summation over monic polynomials $l$ and over integers $n$ up to $g$, and for the non--square contributions the difficulty is to average the non--trivial quadratic character.

\section{The Main Term.}
In this section we will derive an asymptotic formula for
\begin{equation}
\sum_{\substack{n=0 \\ 2|n}}^{g}q^{-n/2}\sum_{\substack{l \ \mathrm{monic} \\ \mathrm{deg}(l)=n/2}}\sum_{\substack{D\in\mathcal{H}_{2g+1,q} \\ (D,l)=1}}1
\end{equation}
which corresponds to the contributions of squares to the average. As in the number field case, the contribution of squares gives the main term of the first moment. The principal result in this section is
\begin{prop}\label{prop:1}With the same notation as in Theorem \ref{thm:mainthm},
\begin{multline}
\sum_{\substack{n=0 \\ 2|n}}^{g}q^{-n/2}\sum_{\substack{l \ \mathrm{monic} \\ \mathrm{deg}(l)=n/2}}\sum_{\substack{D\in\mathcal{H}_{2g+1,q} \\ (D,l)=1}}1\\
=\frac{P(1)}{\zeta_{A}(2)}|D|\left\{([g/2]+1)+\sum_{\substack{P \ \mathrm{monic} \\ \mathrm{irreducible}}}\frac{\mathrm{deg}P}{|P|(|P|+1)-1}\right\}+O(gq^{\frac{3}{2}g}).
\end{multline}
\end{prop}
We will need some preliminary lemmas.

\subsection{Counting Square--free polynomials which are coprime to another monic polynomial.}
We will prove here the following proposition.
\begin{prop}\label{prop:2}
\begin{equation}
\sum_{\substack{D\in\mathcal{H}_{2g+1,q} \\ (D,l)=1}}1=\frac{|D|}{\zeta_{A}(2)\prod_{P|l}(1+|P|^{-1})}+O\left(\sqrt{|D|}\frac{\Phi(l)}{|l|}\right).
\end{equation}
\end{prop}
We will need the following three lemmas.
\begin{lem}\label{lem:4}
Let $V_{d}=\{D\in\mathbb{F}_{q}[x]: D \ \mathrm{monic}, \ \mathrm{deg}(D)=d\}$. Then,
\begin{equation}
\#\{D\in V_{d}: (D,l)=1\}=q^{d}\frac{\Phi(l)}{|l|}.
\end{equation}
\end{lem}
\begin{proof}
\begin{eqnarray}
\#\{D\in V_{d}:(D,l)=1\}& = &\sum_{\substack{D \ \mathrm{monic} \\ \mathrm{deg}(D)=d \\ (D,l)=1}}1=\sum_{\substack{D \ \mathrm{monic} \\ \mathrm{deg}(D)=d}}\sum_{h|(D,l)}\mu(h)\nonumber\\
& = & \sum_{h|l}\mu(h)\sum_{\substack{D \ \mathrm{monic} \\ \mathrm{deg}(D)=d \\ h|D}}1=\sum_{h|l}\mu(h)\sum_{\substack{m \ \mathrm{monic} \\ \mathrm{deg}(m)=d-\mathrm{deg}h}}1\nonumber\\
& = & \sum_{h|l}\mu(h)q^{d-\mathrm{deg}(h)}=q^{d}\sum_{h|l}\frac{\mu(h)}{|h|}\nonumber\\
\label{eq:5.5}& = & q^{d}\prod_{\substack{P \\ P|l}}\left(1-\frac{1}{|P|}\right)=q^{d}\frac{\Phi(l)}{|l|}
\end{eqnarray}
where we used Lemma \ref{lem:2} in \eqref{eq:5.5}.
\end{proof}

\begin{lem}\label{lem:5}
We have,
\begin{equation}
\sum_{\substack{Q \ \mathrm{monic} \\ \mathrm{deg}(Q)>\frac{2g+1}{2} \\ (Q,l)=1}}\frac{\mu(Q)}{|Q|^{2}}\ll q^{-1/2}q^{-g}.
\end{equation}
\end{lem}
\begin{proof}
\begin{eqnarray}
\sum_{\substack{Q \ \mathrm{monic} \\ \mathrm{deg}(Q)>\frac{2g+1}{2} \\ (Q,l)=1}}\frac{\mu(Q)}{|Q|^{2}} & \leq & \sum_{\substack{Q \ \mathrm{monic} \\ \mathrm{deg}(Q)>\frac{2g+1}{2} \\ (Q,l)=1}}\frac{1}{|Q|^{2}}\nonumber\\
& \leq & \sum_{n>\frac{2g+1}{2}}\sum_{\substack{Q \ \mathrm{monic} \\ \mathrm{deg}(Q)=n}}\frac{1}{|Q|^{2}}\nonumber\\
& = & \sum_{n>\frac{2g+1}{2}}\frac{1}{q^{n}}\ll q^{-1/2}q^{-g}.
\end{eqnarray}
\end{proof}

\begin{lem}\label{lem:6}
We have that,
\begin{equation}
\sum_{\substack{Q \ \mathrm{monic} \\ \mathrm{deg}(Q)\leq\frac{2g+1}{2} \\ (Q,l)=1}}\frac{\mu(Q)}{|Q|^{2}}=\frac{1}{\zeta_{A}(2)}\frac{1}{\prod_{P|l}(1-1/|P|^{2})}+O(q^{-1/2}q^{-g}).
\end{equation}
\end{lem}
\begin{proof}

\begin{eqnarray}
\sum_{\substack{Q \ \mathrm{monic} \\ \mathrm{deg}(Q)\leq\frac{2g+1}{2} \\ (Q,l)=1}}\frac{\mu(Q)}{|Q|^{2}} & = & \sum_{\substack{Q \ \mathrm{monic} \\ (Q,l)=1}}\frac{\mu(Q)}{|Q|^{2}}-\sum_{\substack{Q \ \mathrm{monic} \\ \mathrm{deg}(Q)>\frac{2g+1}{2} \\ (Q,l)=1}}\frac{\mu(Q)}{|Q|^{2}}\nonumber\\
& = & \prod_{P\nmid l}\left(1-\frac{1}{|P|^{2}}\right)-\sum_{\substack{Q \ \mathrm{monic} \\ \mathrm{deg}(Q)>\frac{2g+1}{2} \\ (Q,l)=1}}\frac{\mu(Q)}{|Q|^{2}},
\end{eqnarray}
and
\begin{eqnarray}
\prod_{P\nmid l}\left(1-\frac{1}{|P|^{2}}\right) & = & \prod_{P}\left(1-\frac{1}{|P|^{2}}\right)\prod_{P\mid l}\left(1-\frac{1}{|P|^{2}}\right)^{-1}\nonumber\\
& = & \frac{1}{\zeta_{A}(2)}\frac{1}{\prod_{P\mid l}(1-1/|P|^{2})}.
\end{eqnarray}
Thus,
\begin{equation}
\sum_{\substack{Q \ \mathrm{monic} \\ \mathrm{deg}(Q)\leq\frac{2g+1}{2} \\ (Q,l)=1}}\frac{\mu(Q)}{|Q|^{2}}=\frac{1}{\zeta_{A}(2)}\frac{1}{\prod_{P|l}(1-1/|P|^{2})}-\sum_{\substack{Q \ \mathrm{monic} \\ \mathrm{deg}(Q)>\frac{2g+1}{2} \\ (Q,l)=1}}\frac{\mu(Q)}{|Q|^{2}},
\end{equation}
and using the estimate of Lemma \ref{lem:5} proves the result.
\end{proof}

\begin{proof}[Proof of Proposition \ref{prop:2}]
Following the proof of Lemma 4.2 in \cite{BucurIMRN} we have that
\begin{eqnarray}
\sum_{\substack{D\in\mathcal{H}_{2g+1,q} \\ (D,l)=1}}1 &=& \sum_{\substack{D\in V_{2g+1} \\ (D,l)=1}}\sum_{Q^{2}\mid D}\mu(Q)=\sum_{\substack{Q \ \mathrm{monic} \\ \mathrm{deg}(Q)\leq\frac{2g+1}{2} \\ (Q,l)=1}}\mu(Q)\sum_{\substack{D\in V_{2g+1-2\mathrm{deg}(Q)} \\ (D,l)=1}}1\nonumber\\
&=& \sum_{\substack{Q \ \mathrm{monic} \\ \mathrm{deg}(Q)\leq\frac{2g+1}{2} \\ (Q,l)=1}}\mu(Q)\#\{D\in V_{2g+1-2\mathrm{deg}(Q)}:(D,l)=1\}.
\end{eqnarray}
By Lemma \ref{lem:4}, we have,
\begin{eqnarray}
\sum_{\substack{D\in\mathcal{H}_{2g+1,q} \\ (D,l)=1}}1 & = & \sum_{\substack{Q \ \mathrm{monic} \\ \mathrm{deg}(Q)\leq\frac{2g+1}{2} \\ (Q,l)=1}}\mu(Q)q^{2g+1-2\mathrm{deg}(Q)}\frac{\Phi(l)}{|l|}\nonumber\\
& = & |D|\frac{\Phi(l)}{|l|}\sum_{\substack{Q \ \mathrm{monic} \\ \mathrm{deg}(Q)\leq\frac{2g+1}{2} \\ (Q,l)=1}}\frac{\mu(Q)}{|Q|^{2}}.
\end{eqnarray}
Invoking Lemma \ref{lem:6} we obtain,
\begin{eqnarray}
\sum_{\substack{D\in\mathcal{H}_{2g+1,q} \\ (D,l)=1}}1 & = & |D|\frac{\Phi(l)}{|l|}\left(\frac{1}{\zeta_{A}(2)}\frac{1}{\prod_{P|l}(1-1/|P|^{2})}+O(q^{-1/2}q^{-g})\right)\nonumber\\
& = & |D|\frac{\Phi(l)}{|l|}\frac{1}{\zeta_{A}(2)}\frac{1}{\prod_{P|l}(1-1/|P|^{2})}+O\left(|D|\frac{\Phi(l)}{|l|}q^{-1/2}q^{-g}\right),
\end{eqnarray}
and using $\displaystyle{\frac{\Phi(l)}{|l|}=\prod_{P|l}(1-|P|^{-1})}$, we end up with
\begin{equation}
\sum_{\substack{D\in\mathcal{H}_{2g+1,q} \\ (D,l)=1}}1=\frac{|D|}{\zeta_{A}(2)\prod_{P|l}(1+|P|^{-1})}+O\left(\sqrt{|D|}\frac{\Phi(l)}{|l|}\right),
\end{equation}
which proves Proposition \ref{prop:2}.
\end{proof}

\subsection{A Sum Over Monic Polynomials.}
In this section we prove the following two lemmas.

\begin{lem}\label{lem:7}
We have that,
\begin{equation}
\prod_{\substack{P \\ P|l}}(1+|P|^{-1})^{-1}=\sum_{\substack{d \ \mathrm{monic} \\ d|l}}\mu(d)\prod_{P|d}\frac{1}{|P|+1}.
\end{equation}
\end{lem}
\begin{proof}
Obviously,
$$\prod_{\substack{P \\ P|l}}(1+|P|^{-1})^{-1}=\prod_{\substack{P \\ P|l}}\left(1-\frac{1}{|P|+1}\right).$$
Let $P_{1},\ldots,P_{m}$ be the primes that divide $l$.  Then
\begin{multline}
\prod_{\substack{P \\ P|l}}\left(1-\frac{1}{|P|+1}\right)= \left(1-\frac{1}{|P_{1}|+1}\right)\left(1-\frac{1}{|P_{2}|+1}\right)\cdots\left(1-\frac{1}{|P_{m}|+1}\right)\\
=1-\left(\frac{1}{|P_{1}|+1}+\cdots+\frac{1}{|P_{m}|+1}\right)+\left(\frac{1}{|P_{1}|+1}\frac{1}{|P_{2}|+1}+\cdots\right)-\cdots \ \ \ \ \ \ \ \ \ \ \ \ \ \ \ \ \ \ \ \ \ \ \ \ \ \ \ \ \ \ \ \ \ \ \ \ \ \ \ \ \ \ \ \ \ \ \ \ \ \ \ \ \ \ \ \\
=\sum_{\substack{d \ \mathrm{monic} \\ d|l}}\mu(d)\prod_{P|d}\frac{1}{|P|+1}, \ \ \ \ \ \ \ \ \ \ \ \ \ \ \ \ \ \ \ \ \ \ \ \ \ \ \ \ \ \ \ \ \ \ \ \ \ \ \ \ \ \ \ \ \ \ \ \ \ \ \ \ \ \ \ \ \ \ \ \ \ \ \ \ \ \ \ \ \ \ \ \ \ \ \ \ \ \ \ \ \ \ \ \ \ \ \ \ \ \ \ \ \ \ \ \ \ \ \ \ \ \ \ \ \ \ \ \ \ \
\end{multline}
which proves the lemma.
\end{proof}

\begin{lem}\label{lem:8}
We have,
\begin{equation}
\sum_{\substack{l \ \mathrm{monic} \\ \mathrm{deg}(l)=n/2}}\prod_{P|l}(1+|P|^{-1})^{-1}=q^{n/2}\sum_{\substack{d \ \mathrm{monic} \\ \mathrm{deg}(d)\leq n/2}}\frac{\mu(d)}{|d|}\prod_{P|d}\frac{1}{|P|+1}.
\end{equation}
\end{lem}
\begin{proof}
Using Lemma \ref{lem:7} we have,
\begin{eqnarray}
\sum_{\substack{l \ \mathrm{monic} \\ \mathrm{deg}(l)=n/2}}\prod_{P|l}(1+|P|^{-1})^{-1} & = & \sum_{\substack{l \ \mathrm{monic} \\ \mathrm{deg}(l)=n/2}}\sum_{\substack{d \ \mathrm{monic} \\ d|l}}\mu(d)\prod_{P|d}\frac{1}{|P|+1}\nonumber \\
& = & \sum_{\substack{d \ \mathrm{monic} \\ \mathrm{deg}(d)\leq n/2}}\sum_{\substack{l \ \mathrm{monic} \\ \mathrm{deg}(l)=n/2 \\ d|l}}\mu(d)\prod_{P|d}\frac{1}{|P|+1} \nonumber\\
& = & \sum_{\substack{d \ \mathrm{monic} \\ \mathrm{deg}(d)\leq n/2}}\mu(d)\prod_{P|d}\frac{1}{|P|+1}\sum_{\substack{l \ \mathrm{monic} \\ \mathrm{deg}(l)=n/2 \\ d|l}}1\nonumber\\
& = & \sum_{\substack{d \ \mathrm{monic} \\ \mathrm{deg}(d)\leq n/2}}\mu(d)\prod_{P|d}\frac{1}{|P|+1}q^{n/2-\mathrm{deg}(d)}\nonumber\\
& = & q^{n/2}\sum_{\substack{d \ \mathrm{monic} \\ \mathrm{deg}(d)\leq n/2}}\frac{\mu(d)}{|d|}\prod_{P|d}\frac{1}{|P|+1}.\nonumber
\end{eqnarray}
\end{proof}

\subsection{Auxiliary lemmas}
To prove Proposition \ref{prop:1}, which is the main result of this section, we will need some additional lemmas which we will now establish.

\begin{lem}
We have that,
\begin{equation}
\sum_{\substack{n=0 \\ 2|n}}^{g}q^{-n/2}\sum_{\substack{l \ \mathrm{monic} \\ \mathrm{deg}(l)=n/2}}\sqrt{|Q|}\frac{\Phi(l)}{|l|}=\sqrt{|Q|}(1-q^{-1})([g/2]+1).
\end{equation}
\end{lem}
\begin{proof}
\begin{eqnarray}
\sum_{\substack{n=0 \\ 2|n}}^{g}q^{-n/2}\sum_{\substack{l \ \mathrm{monic} \\ \mathrm{deg}(l)=n/2}}\sqrt{|Q|}\frac{\Phi(l)}{|l|} & = & \sqrt{|Q|}\sum_{\substack{n=0 \\ 2|n}}^{g}q^{-n}\sum_{\substack{l \ \mathrm{monic} \\ \mathrm{deg}(l)=n/2}}\Phi(l) \nonumber\\
& = & \sqrt{|Q|}\sum_{\substack{n=0 \\ 2|n}}^{g}(1-q^{-1}), \nonumber
\end{eqnarray}
where we have used Lemma \ref{lem:3} to obtain the last equation.  Hence
$$\sqrt{|Q|}\sum_{\substack{n=0 \\ 2|n}}^{g}(1-q^{-1})=\sqrt{|Q|}(1-q^{-1})\sum_{\substack{n=0 \\ 2|n}}^{g}1,$$
which proves the lemma, since $n$ is even.
\end{proof}
So, from this lemma we can conclude that
\begin{equation}\label{eq:5.20}
\sum_{\substack{n=0 \\ 2|n}}^{g}q^{-n/2}\sum_{\substack{l \ \mathrm{monic} \\ \mathrm{deg}(l)=n/2}}\sqrt{|Q|}\frac{\Phi(l)}{|l|}=O(gq^{g}),
\end{equation}
which is a result that will be of use later.

Using the Euler product formula we can prove the following lemma.

\begin{lem}\label{lem:10}
We have that,
\begin{equation}
\sum_{d \ \mathrm{monic}}\frac{\mu(d)}{|d|}\prod_{P|d}\frac{1}{|P|+1}=\prod_{P}\left(1-\frac{1}{|P|(|P|+1)}\right).
\end{equation}
\end{lem}

There are two additional lemmas which will be important in establishing the formula in Proposition \ref{prop:1}.

\begin{lem}\label{lem:11}
We have that,
\begin{equation}
([g/2]+1)\sum_{\substack{d \ \mathrm{monic} \\ \mathrm{deg}(d)>[g/2]}}\frac{\mu(d)}{|d|}\prod_{P|d}\frac{1}{|P|+1}=O(gq^{-g/2}).
\end{equation}
\end{lem}
\begin{proof}
\begin{eqnarray}
\sum_{\substack{d \ \mathrm{monic} \\ \mathrm{deg}(d)>[g/2]}}\frac{\mu(d)}{|d|}\prod_{P|d}\frac{1}{|P|+1} &\leq& \sum_{\substack{d \ \mathrm{monic} \\ \mathrm{deg}(d)>[g/2]}}\frac{\mu^{2}(d)}{|d|}\prod_{P|d}\frac{1}{|P|}\nonumber\\
& \leq & \sum_{\substack{d \ \mathrm{monic} \\ \mathrm{deg}(d)>[g/2]}}|d|^{-2}=\sum_{h>[g/2]}|d|^{-2}\sum_{\substack{d \ \mathrm{monic} \\ \mathrm{deg}(d)=h}}1 \nonumber\\
& = & \sum_{h>[g/2]}q^{-h}\ll q^{-[g/2]}\ll q^{-g/2}.
\end{eqnarray}
So,
$$([g/2]+1)\sum_{\substack{d \ \mathrm{monic} \\ \mathrm{deg}(d)>[g/2]}}\frac{\mu(d)}{|d|}\prod_{P|d}\frac{1}{|P|+1}\ll gq^{-g/2}.$$
\end{proof}

\begin{lem}\label{lem:12}
We have that,
\begin{equation}
\sum_{\substack{d \ \mathrm{monic} \\ \mathrm{deg}(d)>[g/2]}}\frac{\mu(d)}{|d|}\prod_{P|d}\frac{1}{|P|+1}\mathrm{deg}(d)=O(gq^{-g/2}).
\end{equation}
\end{lem}
\begin{proof}
Using the same reasoning as in Lemma \ref{lem:11}
\begin{eqnarray}
\sum_{\substack{d \ \mathrm{monic} \\ \mathrm{deg}(d)>[g/2]}}\frac{\mu(d)}{|d|}\prod_{P|d}\frac{1}{|P|+1}\mathrm{deg}(d) &\leq& \sum_{\substack{d \ \mathrm{monic} \\ \mathrm{deg}(d)>[g/2]}}\frac{\mu^{2}(d)}{|d|}\prod_{P|d}\frac{1}{|P|}\mathrm{deg}(d)\nonumber\\
& = & \sum_{\substack{d \ \mathrm{monic} \\ \mathrm{deg}(d)>[g/2]}}|d|^{-2}\mathrm{deg}(d)=\sum_{h>[g/2]}\sum_{\substack{d \ \mathrm{monic} \\ \mathrm{deg}(d)=h}}hq^{-2h} \nonumber\\
& = & \sum_{h>[g/2]}hq^{-h}\ll [g/2]q^{-[g/2]}\ll gq^{-g/2}.\nonumber
\end{eqnarray}
\end{proof}

Next we establish the following formula.

\begin{prop}\label{prop:3}
We have that,
\begin{multline}
\sum_{d \ \mathrm{monic}}\frac{\mu(d)}{|d|}\prod_{P|d}\frac{1}{|P|+1}\mathrm{deg}(d)\\
=-\prod_{P}\left(1-\frac{1}{|P|(|P|+1)}\right)\sum_{\substack{P \ \mathrm{monic} \\ \mathrm{irreducible}}}\frac{\mathrm{deg}(P)}{|P|(|P|+1)-1}.
\end{multline}
\end{prop}

\begin{proof}
Let,
\begin{equation}
f(s)=\sum_{d \ \mathrm{monic}}\mathrm{deg}(d)\frac{\mu(d)}{|d|^{s}}\prod_{P|d}\frac{1}{|P|+1}
\end{equation}
and
\begin{equation}
g(s)=\sum_{d \ \mathrm{monic}}\frac{\mu(d)}{|d|^{s}}\prod_{P|d}\frac{1}{|P|+1}.
\end{equation}
A simple calculation shows that
\begin{equation}\label{eq:5.28}
g'(s)=-f(s)\log q
\end{equation}
and by Lemma \ref{lem:10}
\begin{equation}\label{eq:5.29}
g(s)=\prod_{P}\left(1-\frac{1}{|P|^{s}(|P|+1)}\right).
\end{equation}
Computing $g'(s)$ using \eqref{eq:5.29} and the product rule gives us
\begin{equation}\label{eq:5.30}
g'(s)=g(s)\log q\sum_{\substack{P \ \mathrm{monic} \\ \mathrm{irreducible}}}\frac{\mathrm{deg}(P)}{|P|^{s}(|P|+1)-1}.
\end{equation}
Combining \eqref{eq:5.28} and \eqref{eq:5.30} we have that
\begin{equation}
f(s)=-g(s)\sum_{\substack{P \ \mathrm{monic} \\ \mathrm{irreducible}}}\frac{\mathrm{deg}(P)}{|P|^{s}(|P|+1)-1}.
\end{equation}
Putting $s=1$ proves the theorem.
\end{proof}

Now we are ready to give a proof of our main result in this section.
\begin{proof}[Proof of Proposition \ref{prop:1}]
Let
\begin{equation}
B(n,l,D)=\sum_{\substack{n=0 \\ 2|n}}^{g}q^{-n/2}\sum_{\substack{l \ \mathrm{monic} \\ \mathrm{deg}(l)=n/2}}\sum_{\substack{D\in\mathcal{H}_{2g+1,q} \\ (D,l)=1}}1.
\end{equation}
By Proposition \ref{prop:2} we have that
\begin{eqnarray}
B(n,l,D)&=&\frac{|D|}{\zeta_{A}(2)}\sum_{\substack{n=0 \\ 2|n}}^{g}q^{-n/2}\sum_{\substack{l \ \mathrm{monic} \\ \mathrm{deg}(l)=n/2}}\prod_{P|l}(1+|P|^{-1})^{-1}\nonumber\\
&+&O\left(\sum_{\substack{n=0 \\ 2|n}}^{g}q^{-n/2}\sum_{\substack{l \ \mathrm{monic} \\ \mathrm{deg}(l)=n/2}}\sqrt{|D|}\frac{\Phi(l)}{|l|}\right)\nonumber
\end{eqnarray}
and using \eqref{eq:5.20} we can reduce $B(n,l,D)$ to
\begin{equation}
B(n,l,D)=\frac{|D|}{\zeta_{A}(2)}\sum_{\substack{n=0 \\ 2|n}}^{g}q^{-n/2}\sum_{\substack{l \ \mathrm{monic} \\ \mathrm{deg}(l)=n/2}}\prod_{P|l}(1+|P|^{-1})^{-1}+O(gq^{g}).
\end{equation}
Using Lemma \ref{lem:8} we have
\begin{eqnarray}
B(n,l,D) & = & \frac{|D|}{\zeta_{A}(2)}\sum_{\substack{n=0 \\ 2|n}}^{g}q^{-n/2}q^{n/2}\sum_{\substack{d \ \mathrm{monic} \\ \mathrm{deg}(d)\leq n/2}}\frac{\mu(d)}{|d|}\prod_{P|d}\frac{1}{|P|+1}+O(gq^{g})\nonumber\\
& = & \frac{|D|}{\zeta_{A}(2)}\sum_{m=0}^{[g/2]}\sum_{\substack{d \ \mathrm{monic} \\ \mathrm{deg}(d)\leq m}}\frac{\mu(d)}{|d|}\prod_{P|d}\frac{1}{|P|+1}+O(gq^{g})\nonumber\\
& = & \frac{|D|}{\zeta_{A}(2)}\sum_{\substack{d \ \mathrm{monic} \\ \mathrm{deg}(d)\leq [g/2]}}\frac{\mu(d)}{|d|}\prod_{P|d}\frac{1}{|P|+1}\sum_{\mathrm{deg}(d)\leq m\leq[g/2]}1+O(gq^{g})\nonumber\\
& = & \frac{|D|}{\zeta_{A}(2)}\sum_{\substack{d \ \mathrm{monic} \\ \mathrm{deg}(d)\leq [g/2]}}\frac{\mu(d)}{|d|}\prod_{P|d}\frac{1}{|P|+1}([g/2]+1-\mathrm{deg}(d))\nonumber\\
& + & O(gq^{g}).
\end{eqnarray}
Hence,
\begin{eqnarray}
B(n,l,D)& = & \frac{|D|}{\zeta_{A}(2)}\left\{([g/2]+1)\left(\sum_{d \ \mathrm{monic}}\frac{\mu(d)}{|d|}\prod_{P|d}\frac{1}{|P|+1}\right)\right\}\nonumber\\
& - & \frac{|D|}{\zeta_{A}(2)}\left\{([g/2]+1)\left(\sum_{\substack{d \ \mathrm{monic} \\ \mathrm{deg}(d)>[g/2]}}\frac{\mu(d)}{|d|}\prod_{P|d}\frac{1}{|P|+1}\right)\right\}\nonumber\\
& - & \frac{|D|}{\zeta_{A}(2)}\left\{\sum_{d \ \mathrm{monic}}\frac{\mu(d)}{|d|}\prod_{P|d}\frac{1}{|P|+1}\mathrm{deg}(d)\right\}\nonumber\\
& + & \frac{|D|}{\zeta_{A}(2)}\left\{\sum_{\substack{d \ \mathrm{monic} \\ \mathrm{deg}(d)>[g/2]}}\frac{\mu(d)}{|d|}\prod_{P|d}\frac{1}{|P|+1}\mathrm{deg}(d)\right\}+O(gq^{g}).
\end{eqnarray}
Combining Lemmas \ref{lem:10},\ref{lem:11},\ref{lem:12} and Proposition \ref{prop:3} we have,
\begin{eqnarray}
B(n,l,D) & = & \frac{|D|}{\zeta_{A}(2)}([g/2]+1)P(1)+\frac{|D|}{\zeta_{A}(2)}P(1)\sum_{\substack{P \ \mathrm{monic} \\ \mathrm{irreducible}}}\frac{\mathrm{deg}(P)}{|P|(|P|+1)-1}+O(gq^{\frac{3}{2}g})\nonumber\\
& = & \frac{P(1)}{\zeta_{A}(2)}|D|\left\{([g/2]+1)+\sum_{\substack{P \ \mathrm{monic} \\ \mathrm{irreducible}}}\frac{\mathrm{deg}(P)}{|P|(|P|+1)-1}\right\}+O(gq^{\frac{3}{2}g}),\nonumber
\end{eqnarray}
which completes the proof of the proposition.
\end{proof}

\section{Estimating the Contributions of Non--Squares to the Average.}
We will present in this section an estimate for the second term of \eqref{eq:4.3} which allows us to give an asymptotic formula for the first term of \eqref{eq:4.2} where $q\equiv 1\pmod{4}$ is fixed and $g\rightarrow\infty$. Our main result in this section is:
\begin{prop}\label{prop:4}
We have,
\begin{equation}
\sum_{n=0}^{g}q^{-n/2}\sum_{\substack{f \ \mathrm{monic} \\ \mathrm{deg}(f)=n \\ f\neq\Box}}\sum_{D\in\mathcal{H}_{2g+1,q}}\chi_{D}(f)=O\left(2^{g}q^{\frac{3}{2}g+\frac{3}{4}}\right).
\end{equation}
\end{prop}

For this we will need the following lemmas (c.f.~\cite{FR}):

\begin{lem}\label{lem:13}
Let $\chi$ be a nontrivial Dirichlet character modulo $f$. Then for $n<\mathrm{deg}(f)$,
\begin{equation}
\left|\sum_{\mathrm{deg}(B)=n}\chi(B)\right|\leq\binom{\mathrm{deg}(f)-1}{n}q^{n/2}
\end{equation}
(the sum over all monic polynomials of degree $n$).
\end{lem}
\begin{proof}
This is straightforward from the Riemann Hypothesis for function fields. All we need to do is compare the series expansion of $\mathcal{L}(u,\chi)$, which is a polynomial of degree at most $\mathrm{deg}(f)-1$, with the expression in terms of the inverse zeros:
$$\sum_{0\leq n<\mathrm{deg}(f)}\left(\sum_{\mathrm{deg}(B)=n}\chi(B)\right)u^{n}=\prod_{j=1}^{\mathrm{deg}(f)-1}(1-\alpha_{j}u)$$
to get
$$\sum_{\mathrm{deg}(B)=n}\chi(B)=(-1)^{n}\sum_{\substack{S\subset\{1,\ldots,\mathrm{deg}(f)-1\} \\ \#S=n}}\prod_{j\in S}\alpha_{j}$$
and then use $|\alpha_{j}|\leq\sqrt{q}$ for all $j$.
\end{proof}

\begin{rmk}
Note that for $n\geq\mathrm{deg}(f)$ the character sum vanishes.
\end{rmk}

Now we apply this result to quadratic characters.
\begin{lem}\label{lem:14}
If $f\in\mathbb{F}_{q}[x]$ is not a square then
\begin{equation}
\left|\sum_{D\in\mathcal{H}_{2g+1,q}}\chi_{D}(f)\right|\ll q^{g+1/2}2^{\mathrm{deg}f-1}.
\end{equation}
\end{lem}
\begin{proof}
We use \eqref{eq:3.20} to pick out the square--free monic polynomials. Thus the sum over all square--free monic polynomials is given by
\begin{eqnarray}
\sum_{D\in\mathcal{H}_{2g+1,q}}\chi_{D}(f) & = & \sum_{\mathrm{deg}(D)=2g+1}\sum_{A^{2}|D}\mu(A)\left(\frac{D}{f}\right)\nonumber\\
& = & \sum_{\mathrm{deg}(A)\leq g}\mu(A)\left(\frac{A}{f}\right)^{2}\sum_{\mathrm{deg}(B)=2g+1-2\mathrm{deg}(A)}\left(\frac{B}{f}\right).
\end{eqnarray}
To deal with the inner sum, note that $(\bullet/f)$ is a nontrivial character since $f$ is not a square, so we can use Lemma \ref{lem:13} to obtain
$$\left|\sum_{\mathrm{deg}(B)=2g+1-2\mathrm{deg}(A)}\left(\frac{B}{f}\right)\right|\leq\binom{\mathrm{deg}(f)-1}{2g+1-2\mathrm{deg}(A)}q^{g+1/2-\mathrm{deg}(A)}$$
if $2g+1-2\mathrm{deg}(A)<\mathrm{deg}(f)$. The sum is zero otherwise. Hence we have
\begin{eqnarray}
\left|\sum_{D\in\mathcal{H}_{2g+1,q}}\chi_{D}(f)\right|& \leq & \sum_{\mathrm{deg}(A)\leq g}\left|\sum_{\mathrm{deg}(B)=2g+1-2\mathrm{deg}(A)}\left(\frac{B}{f}\right)\right|\nonumber\\
& \leq & \sum_{g+1/2-(\mathrm{deg}(f)/2)<\mathrm{deg}(A)\leq g}\binom{\mathrm{deg}(f)-1}{2g+1-2\mathrm{deg}(A)}q^{g+1/2-\mathrm{deg}(A)}\nonumber\\
& = & q^{g+1/2}\sum_{g+1/2-(\mathrm{deg}(f)/2)<j\leq g}\binom{\mathrm{deg}(f)-1}{2g+1-2j}\leq2^{\mathrm{deg}(f)-1}q^{g+1/2}.
\end{eqnarray}
This completes the proof of Lemma \ref{lem:14}.
\end{proof}
\begin{proof}[Proof of Proposition \ref{prop:4}]
Using Lemma \ref{lem:14} we have that,
\begin{eqnarray}
\sum_{n=0}^{g}q^{-n/2}\sum_{\substack{f \ \mathrm{monic} \\ \mathrm{deg}(f)=n \\ f\neq\Box}}\sum_{D\in\mathcal{H}_{2g+1,q}}\chi_{D}(f) & \leq & \sum_{n=0}^{g}q^{-n/2}\sum_{\substack{f \ \mathrm{monic} \\ \mathrm{deg}(f)=n}}2^{\mathrm{deg}f-1}q^{g+1/2}\nonumber\\
& = & \sum_{n=0}^{g}q^{-n/2}2^{n-1}q^{g+1/2}q^{n}\nonumber\\
& \ll & q^{g}\sum_{n=0}^{g}(q^{1/2}2)^{n}\nonumber\\
& \ll & q^{g}(2q^{1/2})^{g+1}\nonumber\\
& \ll & q^{\frac{3}{2}g+\frac{3}{4}}2^{g}.
\end{eqnarray}
\end{proof}

\section{Proof of the Main Theorem}
\begin{proof}[Proof of Theorem \ref{thm:mainthm}]
Now we are in a position to prove Theorem \ref{thm:mainthm}. For this we make use of Proposition \ref{prop:1} and Proposition \ref{prop:4}, which give us
\begin{multline}\label{eq:7.1}
\sum_{D\in\mathcal{H}_{2g+1,q}}\sum_{n=0}^{g}\sum_{\substack{f_{1} \ \mathrm{monic} \\ \mathrm{deg}(f_{1})=n}}\chi_{D}(f_{1})q^{-n/2}\\
=\frac{P(1)}{\zeta_{A}(2)}|D|\left\{([g/2]+1)+\sum_{P}\frac{\mathrm{deg}(P)}{|P|(|P|+1)-1}\right\}+O\left(2^{g}q^{\frac{3}{2}g+\frac{3}{4}}\right).
\end{multline}
For the dual sum in \eqref{eq:4.2} we get, similarly, that
\begin{multline}\label{eq:7.2}
\sum_{D\in\mathcal{H}_{2g+1,q}}\sum_{m=0}^{g-1}\sum_{\substack{f_{2} \ \mathrm{monic} \\ \mathrm{deg}(f_{2})=m}}\chi_{D}(f_{2})q^{-m/2}\\
=\frac{P(1)}{\zeta_{A}(2)}|D|\left\{\left(\left[\frac{g-1}{2}\right]+1\right)+\sum_{P}\frac{\mathrm{deg}(P)}{|P|(|P|+1)-1}\right\}+O\left(2^{g}q^{\frac{3}{2}g+\frac{3}{4}}\right).
\end{multline}
So adding \eqref{eq:7.1} with \eqref{eq:7.2} we see that,
\begin{multline}
\sum_{D\in\mathcal{H}_{2g+1,q}}\mathcal{L}(q^{-1/2},\chi_{D})\\
=\frac{P(1)}{2\zeta_{A}(2)}|D|\left\{\log_{q}|D|+1+4\sum_{P}\frac{\mathrm{deg}(P)}{|P|(|P|+1)-1}\right\}+O\left(2^{g}q^{\frac{3}{2}g+\frac{3}{4}}\right)
\end{multline}
and using the fact that $|D|=q^{2g+1}$ we have precisely the statement of Theorem \ref{thm:mainthm}.
\end{proof}
The corollary is immediate using \eqref{eq:3.17} and computing the limit as $g\rightarrow\infty$.

\section{Acknowledgments}
We would like to thank Professor Brian Conrey for pointing out the reference \cite{J}. The authors also wish to thank Chantal David, P\"{a}r Kurlberg, Ze\'{e}v Rudnick, Nina Snaith and Trevor Wooley for helpful and interesting discussions. 

The authors also wish to thank an anonymous referee for his comments, which were so helpful in improve the original manuscript.

\end{document}